\documentclass[11pt,a4paper]{article}
 
\usepackage{fullpage}
\usepackage[utf8]{inputenc}
\usepackage{amssymb}
\usepackage{amsmath}
\usepackage{amsthm}
\usepackage[round]{natbib}
\usepackage{url}
\usepackage{graphicx}
\usepackage{color}
\usepackage[colorlinks=true,linkcolor=blue,citecolor=blue,pdfborder={0 0 0}]{hyperref}
\usepackage{bm}

\usepackage{paralist}
\newcommand{\eps}{\varepsilon}
\newcommand{\N}{\mathbb{N}}
\newcommand{\R}{\mathbb{R}}
\newcommand{\Q}{\mathbb{Q}}

\newcommand{\Ac}{\mathcal{A}}

\newcommand{\Dc}{\mathcal{D}}
\newcommand{\Fc}{\mathcal{F}}
\newcommand{\Hc}{\mathcal{H}}

\newcommand{\Db}{\mathbb{D}}
\newcommand{\Gb}{\mathbb{G}}

\newcommand{\G}{\mathbb{G}}
\newcommand{\X}{\mathcal{X}}
\newcommand{\W}{\mathcal{W}}

\newcommand{\Ex}{\mathbb{E}}
\newcommand{\Exp}{\mathbb{E}}
\newcommand{\Var}{\mathrm{Var}}
\newcommand{\1}{\mathbf{1}}

\renewcommand{\Pr}{\mathbb{P}}
\newcommand{\BL}{\mathrm{BL}}

\newcommand{\p}{\overset{\Pr}{\to}}
\newcommand{\op}{\overset{\Pr^*}{\to}}
\newcommand{\as}{\overset{\mathrm{a.s.}}{\longrightarrow}}

\newcommand{\scs}{\scriptscriptstyle}
\newcommand{\ex}[1]{\text{\hspace{-5pt}\raisebox{3pt}{$\scs #1$}}}

\theoremstyle{plain}
\newtheorem{prop}{Proposition}[section]

\newtheorem{cor}[prop]{Corollary}
\newtheorem{lem}[prop]{Lemma}
\newtheorem{cond}[prop]{Condition}

\newtheorem{innercustomlem}{Lemma}
\newenvironment{customlem}[1]
  {\renewcommand\theinnercustomlem{#1}\innercustomlem}
  {\endinnercustomlem}
  
\theoremstyle{remark}

\numberwithin{equation}{section}

\title{A note on conditional versus joint unconditional weak convergence in bootstrap consistency results}

\author{Axel B\"ucher\,\footnote{Ruhr-Universit\"at Bochum,
Fakult\"at f\"ur Mathematik, 
Universit\"atsstr.~150, 44780 Bochum, Germany. 
{E-mail:} \texttt{axel.buecher@rub.de}}\;~ and Ivan Kojadinovic\,\footnote{CNRS / Universit\'e de Pau et des Pays de l'Adour, Laboratoire de math\'ematiques et applications -- IPRA, UMR 5142, B.P. 1155, 64013 Pau Cedex, France.
{E-mail:} \texttt{ivan.kojadinovic@univ-pau.fr}}
}

\begin{document}
\maketitle

\begin{abstract}
The consistency of a bootstrap or resampling scheme is classically validated by weak convergence of conditional laws. However, when working with stochastic processes in the space of bounded functions and their weak convergence in the Hoffmann-J\o rgensen sense, an obstacle occurs: due to possible non-measurability, neither laws nor conditional laws are well-defined. Starting from an equivalent formulation of weak convergence based on the bounded Lipschitz metric, a classical circumvent is to formulate bootstrap consistency in terms of the latter distance between what might be called a \emph{conditional law} of the (non-measurable) bootstrap process and the law of the limiting process. The main contribution of this note is to provide an equivalent formulation of bootstrap consistency in the space of bounded functions which is more intuitive and easy to work with. Essentially, the equivalent formulation consists of (unconditional) weak convergence of the original process jointly with two bootstrap replicates. As a by-product, we provide two equivalent formulations of bootstrap consistency for statistics taking values in separable metric spaces: the first in terms of (unconditional) weak convergence of the statistic jointly with its bootstrap replicates, the second in terms of convergence in probability of the empirical distribution function of the bootstrap replicates. Finally, the asymptotic validity of bootstrap-based confidence intervals and tests is briefly revisited, with particular emphasis on the, in practice unavoidable, Monte Carlo approximation of conditional quantiles.

\medskip


\medskip

\end{abstract}

\section{Introduction}

It is not uncommon in statistical problems that the limiting distribution of a statistic of interest be intractable. To carry out inference on the underlying quantity, one possibility consists of using a \emph{bootstrap} or \emph{resampling scheme}. Ideally, prior to its use, its \emph{consistency} or \emph{asymptotic validity} should be mathematically demonstrated. For a real or vector-valued statistic~$\bm S_n$ (or, more generally, a statistic taking values in a separable metric space $\Db$), the latter classically consists of establishing weak convergence of certain conditional laws (assuming that these are well-defined; see, e.g., \citealp{Fad85} or Section~6 in \citealp{Kal02}). Specifically, a resampling scheme can be considered asymptotically consistent if an appropriate distance between the conditional distribution of a \emph{bootstrap replicate} of $\bm S_n$ \emph{given} the available observations and  the distribution of $\bm S_n$ is shown to converge to zero in probability; see, for instance, \cite{BicFre81},  \citet[Chapter~23]{Van98}, \cite{Hor01} and the references therein, or Assertions~$(c)$ and~$(e)$ in Lemma~\ref{lem:uncond:cond} below. A first contribution of this note is to show that, under minimal conditions, the aforementioned convergence of conditional laws is actually equivalent to the (unconditional) weak convergence of $\bm S_n$ jointly with two bootstrap replicates to independent copies of the same limit. As pointed out by a referee, the proof of this result relies on a key idea dating back to \cite{Hoe52}, which has been used to derive quite similar statements since then; see, for instance, Lemma~4.1 in \cite{DumZer13} and the additional references given in Section~\ref{sec:equiv:cond:uncond}. Furthermore, we provide an interesting third equivalent formulation of the consistency of a bootstrap for $\bm S_n$. It roughly states that the distance between the empirical distribution of the bootstrap replicates and the unobservable distribution of $\bm S_n$ converges in probability to zero as the number of replicates and the sample size increase (see also \citealp{BerLecMil87}, Section~4, for a similar result). The latter is particularly meaningful given that most applications of resampling involve at some point approximating the unobservable distribution of $\bm S_n$ by the empirical distribution of a finite number of bootstrap replicates. 

In many situations, the $\Db$-valued statistic of interest $\bm S_n$ is a ``sufficiently smooth'' functional of a certain stochastic process $\Gb_n$ belonging to the space $\ell^\infty(T)$ of bounded functions defined on some arbitrary  set $T$ (think of the general empirical process, for instance, as defined in Chapter~2 of \citealp{VanWel96}). Note that $\ell^\infty(T)$, when equipped with the supremum distance, is in general neither separable nor complete, and that $\Gb_n$ is usually allowed to be non-measurable as well. As a consequence, neither laws nor conditional laws are well-defined in general, which complicates the theoretical analysis of bootstraps for $\Gb_n$. Following \cite{GinZin90}, the consistency of a resampling scheme is then commonly defined by the requirement that the bounded Lipschitz distance between the candidate limiting law and a suitable adaptation of what might be called a conditional law of the bootstrap replicate (even though the latter does not exist in the classical sense due to non-measurability) converges to zero in outer probability. For instance, for the general empirical process based on independent and identically distributed observations, such an investigation is carried out in \cite{PraWel93} \citep[see also][Section~3.6]{VanWel96} for the so-called \emph{empirical bootstrap} and various other \emph{exchangeable bootstraps}. The appeal of working at the stochastic process level then arises from the fact that such bootstrap consistency results can be transferred to the $\Db$-valued statistic level (often, $\Db=\R^d$) by means of appropriate extensions of the continuous mapping theorem and the functional delta method. 

It may however be argued that the aforementioned generalization of the classical conditional formulation of  bootstrap consistency is unintuitive and complicated to use given the subtlety of the underlying mathematical concepts (in particular, relying on ``conditional laws'' of non-measurable maps). The latter seems all the more true for instance for empirical processes based on estimated or serially dependent observations \citep[see, e.g.,][]{RemSca09,Seg12,BucKoj16}. The main contribution of this note is to show that the $\Db$-valued results from Section~\ref{sec:equiv:cond:uncond} continue to hold for stochastic processes with bounded sample paths: the ``conditional'' formulation is actually equivalent to the (unconditional) weak convergence of the initial stochastic process jointly with two bootstrap replicates. From a practical perspective, using the latter unconditional formulation may have two important advantages. First and most importantly, it may be easier to prove in certain situations than the conditional formulation. For this reason, it was for instance used, as explained above, for empirical processes based on estimated or serially dependent observations; see also Section~\ref{sec:extension} below for additional references. Second, the unconditional formulation may be transferable to the statistic level for a slightly larger class of functionals of the stochastic process under consideration. The latter follows for instance from the fact that continuous mapping theorems \emph{for the bootstrap}, that is, adapted to the conditional formulation, require more than just continuity of the map that transforms the stochastic process into the statistic of interest \cite[see, e.g.,][Section~10.1.4]{Kos08}. Furthermore, there does not seem hitherto to exist an \emph{extended} continuous mapping theorem \citep[see, e.g.,][Theorem 1.11.1]{VanWel96} \emph{for the bootstrap}. Once the unconditional formulation is transferred to a separable metric space $\Db$ (with $\Db$ being typically  $\R^d$), the classical conditional statement immediately follows by the equivalence at the $\Db$-valued statistic level mentioned above. Finally, let us mention that the equivalence at the stochastic process level is well-known for the special case of \emph{multiplier central limit theorems (CLTs)} for the general empirical process based on i.i.d.\ observations using results of \citet[Section~2.9]{VanWel96} (note that multiplier CLTs are sometimes also referred to as multiplier or weighted bootstraps; see, e.g., \citealp{Kos08}, \cite{CheHua10} and the references therein). As such, our proven equivalence at the stochastic process level can be seen as an extension of the latter work.

As an illustration of our results, we revisit the fact that bootstrap consistency implies that bootstrap-based confidence intervals are asymptotically valid in terms of coverage and that bootstrap-based tests hold their level asymptotically; see, for instance, \citet[Lemma~23.3]{Van98} for a related result and \citet[Sections~3.3 and 3.4]{Hor01} for more specialized and deeper results. In particular, we provide results which explicitly take into account that (unobservable) conditional quantiles must be approximated by Monte Carlo in practice. 

Finally, we would like to stress that the asymptotic results in this note are all of first order. Higher order correctness of a resampling scheme (usually considered for real-valued statistics) may still be important in small samples. The reader is referred to \cite{Hal92} for more details. 

This note is organized as follows. The equivalence between the aforementioned formulations of asymptotic validity of bootstraps of statistics taking values in separable metric spaces is proved in Section~\ref{sec:equiv:cond:uncond}. Section~\ref{sec:extension} states conditions under which the results of Section~\ref{sec:equiv:cond:uncond} extend to stochastic processes with bounded sample paths. In Section~\ref{sec:conf:int:tests}, it is formally verified that, as expected, bootstrap consistency implies asymptotic validity of bootstrap-based confidence intervals and tests. A summary of results and concluding remarks are given in the last section. 

In the rest of the document, the arrow `$\leadsto$' denotes weak convergence, while the arrows `$\overset{\scs \mathrm{a.s.}}{\longrightarrow}$' and `$\overset{\scs\Pr}{\to}$' denote almost sure convergence and convergence in probability, respectively.

\section{Equivalent statements of bootstrap consistency in separable metric spaces}
\label{sec:equiv:cond:uncond}

The generic setup considered in this section is as follows. The available data will be denoted by~$\bm X_n$. Apart from measurability, no assumptions are made on $\bm X_n$, but it is instructive to think of $\bm X_n$ as an $n$-tuple of multivariate observations which may possibly be serially dependent. Let~$(\Db,d)$ denote a separable metric space. We are interested in approximating the law of some $\Db$-valued statistic computed from $\bm X_n$, denoted by $\bm S_n=\bm S_n(\bm X_n)$.  $\Db$-valued \emph{bootstrap replicates} of $\bm S_n$, on which inference could be based, will be denoted by $\bm S_{n}^{\scs (1)} = \bm S_{n}^{\scs (1)}(\bm X_n,\bm W_n^{\scs (1)})$, $\bm S_{n}^{\scs (2)} = \bm S_{n}^{\scs (2)}(\bm X_n,\bm W_n^{\scs (2)})$, \dots, where $\bm W_n^{\scs (1)}$, $\bm W_n^{\scs (2)}$, \dots, typically $\R$-valued, are identically distributed and represent additional sources of randomness such that $\bm S_{n}^{\scs (1)}, \bm S_{n}^{\scs (2)}, \dots$ are independent conditionally on $\bm X_n$.  

The previous setup is general enough to encompass most if not all types of resampling procedures. For instance, when $\Db = \R^d$, the classical \emph{empirical} (multinomial) bootstrap of \cite{Efr79} based on resampling with replacement from some original i.i.d.\ data set $\bm X_n=(X_1, \dots, X_n)$ can be obtained by letting the $\bm W_n^{\scs (i)}=(W_{n1}^{\scs (i)}, \dots, W_{nn}^{\scs (i)})$ be i.i.d.\ multinomially distributed with parameter $(n,1/n, \dots, 1/n)$. Indeed, for fixed $i\in\N$, the sample $\bm X_n^*=(X_1^*, \dots, X_n^*)$ constructed by including the $j$th original observation $X_{j}$ exactly $W_{nj}^{\scs (i)}$ times, $j \in \{1, \dots, n\}$, may be identified with a sample being drawn with replacement from the original observations. Many other resampling schemes are included as well: \emph{block} bootstraps for time series such as the one of \cite{Kun89}, (possibly dependent) \emph{multiplier} (or \emph{weighted}, \emph{wild}) bootstraps \citep[see, e.g.,][]{Sha10} or the \emph{parametric} bootstrap \citep[see, e.g.,][]{StuGonPre93,GenRem08}. For all but the last mentioned resampling scheme,  $\bm W_n^{\scs (1)}$, $\bm W_n^{\scs (2)}$, \dots, could be interpreted as i.i.d.\ vectors of \emph{bootstrap weights}, independent of $\bm X_n$. Several examples of such weights when $\bm X_n$ corresponds to $n$ i.i.d.\ observations are given for instance in \citet[Section~3.6.2]{VanWel96}. 

The previous setup is formally summarized in the following assumption. Recall the notions of conditional independence and regular conditional distribution; see, e.g., \cite{Kal02}, Section~6.

\begin{cond}[$\Db$-valued resampling mechanism] \label{cond:Db}
Let $(\Db,d)$ denote a separable metric space equipped with the Borel sigma field $\Dc$, and let $(\Omega, \Ac, \Pr)$ denote a probability space. For $n\in\N$, let $\bm X_n:\Omega \to \mathcal X_n$ be a random variable in some measurable space $\mathcal X_n$. Furthermore, let $\bm W_n^{\scs (i)}:\Omega \to \mathcal W_n$, $i\in\N$, denote identically distributed random variables in some measurable space $\mathcal W_n$ and let $\bm S_{n}^{\scs (i)} = \bm S_{n}^{\scs (i)}(\bm X_n,\bm W_n^{\scs (i)})$, $i\in\N$, be  $\Db$-valued statistics (to be considered as bootstrap replicates of some $\Db$-valued statistic $\bm S_n=\bm S_n(\bm X_n)$) that are independent conditionally on $\bm X_n$. Finally, assume that $\Pr(\bm S_n^{\scs (1)} \in \cdot \mid \bm X_n)$ has a regular version, denoted by $\Pr^{ \bm S_n\ex{(1)} \mid \bm X_n}:\mathcal X_n \times \Dc \to \R$ and called the (regular) conditional distribution of $\bm S_n^{\scs (1)}$ given $\bm X_n$.
\end{cond}

The last assumption in the previous condition concerning the existence of the conditional distribution of $\bm S_n^{\scs  (1)}$ given $\bm X_n$ is automatically satisfied if there exists a possibly different metric~$e$ on $\Db$ which is equivalent to $d$ such that $(\Db, e)$ is complete. In that case, $\Db$ is a Borel space, see Theorem~A1.2 in \cite{Kal02}, and the assertion follows from Theorem~6.3 in that reference. The existence of the aforementioned conditional distribution can also be guaranteed if the underlying probability space has a product structure, that is, if $\Omega=\Omega_0 \times \Omega_1 \times \cdots$ with probability measure $\Pr=\Pr_0 \otimes \Pr_1\otimes \cdots$, where $\Pr_i$ denotes the probability measure on~$\Omega_i$, such that, for any $\omega \in \Omega$, $\bm X_n(\omega)$ only depends on the first coordinate of $\omega$ and $\bm W_n^{\scs (i)}(\omega)$ only depends on the $(i+1)$-coordinate of~$\omega$, implying in particular that $\bm X_n, \bm W_n^{\scs (1)}, \bm W_n^{\scs (2)}, \dots$ are independent. In that case, it can readily be checked by Fubini's theorem that $(\bm x_n, A) \mapsto  \Pr_1(\bm S_{n}^{\scs (1)}(\bm x_n,\bm W_n^{\scs (1)}) \in A)$ defines a regular version of the conditional distribution of $\bm S_n^{\scs (1)}$ given $\bm X_n$.

In a related way, for arbitrary real-valued functions $h$ such that $\Ex|h(\bm S_n^{\scs (1)})|<\infty$, conditional expectations $\Ex \{ h(\bm S_n^{\scs (1)})\mid \bm X_n \}$ are always to be understood as integration of $h(\bm S_n^{\scs (1)})$ with respect to $\Pr^{\bm S_n\ex{(1)} \mid \bm X_n}$ (\citealp{Kal02}, Theorem 6.4).

Lemma~\ref{lem:uncond:cond} below is one of the main result of this note and essentially shows that the unconditional weak convergence of a statistic jointly with two of its bootstrap replicates is equivalent to the convergence in probability of the conditional law of a bootstrap replicate. The latter (with convergence in probability possibly replaced by almost sure convergence) is the classical mathematical definition of the asymptotic validity of a resampling scheme. A further equivalent formulation, of interest for applications, is also provided. Parts of these equivalences can also be found in \cite{DumZer13}, Lemma 4.1, relying on ideas put forward in \cite{Hoe52} and also exploited in \cite{Rom89} and \cite{ChuRom13}.

Recall that the bounded Lipschitz metric $d_\BL$ between probability measures $P,Q$ on a separable metric space $(\Db,d)$ equipped with the Borel sigma field $\Dc$ is defined by
\[
d_\BL(P,Q) = \sup_{f \in \BL_1(\Db)} \big| \textstyle \int f dP - \int f dQ \big|,
\]
where $\BL_1(\Db)$ denotes the set of functions $h:\Db \to [-1,1]$ such that $|h(x) - h(y)| \leq d(x,y)$ for all $x,y \in \Db$. Moreover, recall the Kolmogorov distance $d_K$ between  probability measures $P,Q$ on $\R^d$, defined by
\[
d_K(P,Q) = \sup_{\bm x\in\R^d} \big| P \{ (-\bm \infty, \bm x] \} - Q \{ (-\bm \infty, \bm x] \} \big| .
\]
Finally, denote the empirical distribution of the sample $\bm S_n^{\scs (1)}, \dots, \bm S_n^{\scs (M)}$ by
\[
\hat \Pr_M^{\bm S_n} = \frac1M \sum_{i=1}^M \delta_{\bm S_n^{(i)}}.
\]

\begin{lem}[Equivalence of unconditional and conditional formulations] \label{lem:uncond:cond} Suppose that Condition~\ref{cond:Db} is met. Assume further that $\bm S_n =\bm S_n(\bm X_n)$ converges weakly to some random variable $\bm S$ in $\Db$. Then, the following four assertions are equivalent: 
\begin{align*}
&(a) &   &
\Pr^{(\bm S_n, \bm S_n^{(1)}, \bm S_n^{(2)})} \leadsto \Pr^{\bm S} \otimes \Pr^{\bm S} \otimes \Pr^{\bm S} 
\quad \text{as $n\to\infty$,} & & \\
&(b) &  &
\Pr^{(\bm S_n, \bm S_n^{(1)}, \dots, \bm S_n^{(M)})} \leadsto (\Pr^{\bm S})^{\otimes (M+1)}
\quad \text{as $n\to\infty$ and for any $M\ge 2$,} &  & \\
&(c) &  &
d_{\BL}\left(\Pr^{ \bm S_n^{(1)} \mid \bm X_n}, \Pr^{\bm S_n}\right)  \p 0 
\quad \text{as $n\to\infty$,} &  & \\
&(d) &  &
d_\BL\left(\hat \Pr_M^{\bm S_n} , \Pr^{\bm S_n} \right)  \p 0
\quad \text{as $n,M\to\infty$}. &  & 
\intertext{If, additionally,  $\Db=\R^d$ and the (cumulative) distribution function (d.f.) of $\bm S$ is continuous, then the preceding four assertions are also equivalent to
}
&(e) &   &
d_{K}\left(\Pr^{ \bm S_n^{(1)} \mid \bm X_n}, \Pr^{\bm S_n}\right) \p 0
\quad \text{as $n\to\infty$,} & & \\
&(f) &  &
d_K\left( \hat \Pr_M^{\bm S_n} , \Pr^{\bm S_n} \right)  \p 0
\quad \text{as $n,M\to\infty$}. &  & 
\end{align*}
\end{lem}

Before providing a proof of this lemma, let us give an interpretation of the assertions. The intuition behind Assertions~$(a)$ and $(b)$ is that a resampling scheme should be considered consistent if the bootstrap replicates $\bm S_{n}^{\scriptscriptstyle (1)}, \bm S_{n}^{\scriptscriptstyle (2)}, \dots$ behave approximately as independent copies of~$\bm S_n$, the more so that $n$ is large.  Assertions~$(c)$ and $(e)$ translate mathematically the idea that a resampling scheme should be considered valid if the distribution of a bootstrap replicate given the data is close to the distribution of the original statistic $\bm S_n$, the more so that $n$ is large. Assertions~$(d)$ and $(f)$ can be regarded as empirical analogues of Assertions~$(c)$ and~$(e)$, respectively: the unobservable conditional law of a bootstrap replicate is replaced by the empirical law of a sample of $M$ bootstrap replicates, providing an approximation of the law of $\bm S_n$ that improves as $n,M$ increase.

Assertions~$(c)$ and $(e)$ are known to hold for many statistics and resampling schemes, possibly as a consequence of general consistency results such as the one of \cite{BerDuc91} \citep[see also][Section~2.1]{Hor01}. Assertions~$(a)$ and $(b)$ are substantially less frequently encountered in the literature and appear mostly as a consequence of similar assertions at a stochastic process level; see Lemma~\ref{lem:proc} in Section~\ref{sec:extension} and the references therein.

Let us finally turn to the proof of Lemma~\ref{lem:uncond:cond}. The latter is in fact a corollary of the following, slightly more general lemma which does not rely on the additional assumption that $\bm S_n$ converges weakly. 

\begin{lem}\label{lem:uncond:cond2} Suppose that Condition~\ref{cond:Db} is met and let $Q$ be a fixed probability measure on $(\Db, \Dc)$. Then, the following four assertions are equivalent:
\begin{align*}
&(a) &   &
\Pr^{(\bm S_n^{(1)}, \bm S_n^{(2)})} \leadsto Q \otimes Q
\quad \text{as $n\to\infty$,} & & \\
&(b) &  &
\Pr^{(\bm S_n^{(1)}, \dots, \bm S_n^{(M)})} \leadsto Q^{\otimes M}
\quad \text{as $n\to\infty$ and for any $M\ge 2$,} &  & \\
&(c) &  &
d_{\BL}\left(\Pr^{ \bm S_n^{(1)} \mid \bm X_n}, Q\right)  \p 0 
\quad \text{as $n\to\infty$,} &  & \\
&(d) &  &
d_\BL\left(\hat \Pr_M^{\bm S_n} , Q \right)  \p 0
\quad \text{as $n,M\to\infty$}. &  & 
\intertext{If, additionally,  $\Db=\R^d$ and the d.f.\ of $Q$ is continuous, then the preceding four assertions are also equivalent to
}
&(e) &   &
d_{K}\left(\Pr^{ \bm S_n^{(1)} \mid \bm X_n}, Q \right) \p 0
\quad \text{as $n\to\infty$,} & & \\
&(f) &  &
d_K\left( \hat \Pr_M^{\bm S_n} , Q  \right)  \p 0
\quad \text{as $n,M\to\infty$}. &  & 
\end{align*}
\end{lem}

The proof of this lemma will in turn be based on the following two possibly well-known lemmas about metrizing weak convergence in separable metric spaces. Note that the results are stated in terms of \emph{nets} which generalize sequences (see, e.g., \citealp{VanWel96}, Section~1.1) in order to account for the net convergences in Assertions~$(d)$ and $(f)$ of the two preceding lemmas. 

For sequences, the forthcoming assertions regarding the Kolmogorov distance can for instance be found in \cite{Van98}, see Lemma 2.11 and  Problem 23.1, while the assertions regarding the bounded Lipschitz metric can be found in \cite{Dud02}, Theorem 11.3.3, for the non-random version (see Lemma~\ref{lem:bl1} below) and in \cite{DumZer13}, Section~2, for the random one (see Lemma~\ref{lem:bl1c} below). Detailed proofs are provided in the supplementary material for the sake of completeness. 

\begin{lem} \label{lem:bl1}
Suppose that $(\Db,d)$ is a separable metric space and let $P_\alpha$ be a net of probability measures on $(\Db,\mathcal D)$, where $\Dc$ denotes the Borel sigma field. Then $P_\alpha \leadsto P$ if and only if $d_\BL(P_\alpha,P) \to 0$. If $\Db=\R^d$ and if the d.f.\ of $P$ is continuous, we also have equivalence to $d_K(P_\alpha,P) \to 0$.
\end{lem}

A random probability measure $\hat P$ on a separable metric space $(\Db, d)$ is a mapping from some probability space $(\Omega, \mathcal A, \Pr)$ into the set of Borel probability measures on $(\Db, \mathcal D)$  such that $ \int f d\hat P$ considered as a function from $\Omega$ to $\R$ is measurable for any bounded and continuous function $f$ on $\Db$ (see, e.g., \citealp{DumZer13}, Section~2). Note that, under Condition~\ref{cond:Db},  $\omega \mapsto \Pr^{\bm S_n\ex{(1)} \mid \bm X_n}(\bm X_n(\omega), \cdot)$ is a sequence of such random probability measures. 

\begin{lem} \label{lem:bl1c}
Suppose that $(\Db,d)$ is a  separable metric space and let $(\hat P_\alpha)_\alpha$ denote a net of random probability measures on $(\Db,\Dc)$ defined on a probability space $(\Omega, \Ac, \Pr)$. Then, 
\begin{align} \label{eq:rw}
\int fd\hat P_\alpha \p \int f dP 
\end{align}
for any $f$ bounded and Lipschitz continuous if and only if $d_\BL(\hat P_\alpha,P) \to 0$ in probability. Further, $d_\BL(\hat P_\alpha,P)$, considered as a map from $\Omega$ to $\R$, is measurable.

If $\Db=\R^d$ and if the d.f.\ of $P$ is continuous, then \eqref{eq:rw} is also equivalent to $d_K(\hat P_\alpha,P) \to 0$ in probability, and $d_K(\hat P_\alpha,P)$ is measurable as well.
\end{lem}

We can now prove Lemma~\ref{lem:uncond:cond2}. 

\begin{proof}[Proof of Lemma~\ref{lem:uncond:cond2}]
We begin by showing the equivalence between $(a)$, $(b)$, $(c)$ and $(d)$. Note that, even though the equivalence between $(a)$ and $(c)$ is almost identical to Lemma 4.1 of \cite{DumZer13}, we provide a self-contained proof to ease readability.

$(b) \Rightarrow (a)$: trivial.

$(a) \Rightarrow (c)$: by Lemma~\ref{lem:bl1}, we only need to show that 
\begin{equation*}
\int f d \Pr^{ \bm S_n^{(1)} \mid \bm X_n} \p \int fd Q  \quad \text{as } n \to \infty,
\end{equation*}
for all bounded and Lipschitz continuous $f$. As in the proof of Lemma~4.1 in \cite{DumZer13}, we can even prove $L^2$-convergence. Let $\bm S$ denote a random variable with distribution $Q$. Then, by the law of iterated expectation,
\begin{multline*}
\textstyle \Exp \Big\{ \Big(\int f d \Pr^{ \bm S_n^{(1)} \mid \bm X_n} -  \int f dQ \Big)^2 \Big\}
=
\Exp \Big( \big[ \Exp \{ f(\bm S_n^{(1)}) \mid \bm X_n \} - \Exp \{ f(\bm S) \} \big]^2 \Big) \\
= \Exp \left( \big[ \Exp \{ f(\bm S_n^{(1)}) \mid \bm X_n \} \big]^2 \right) -  2   \Exp \{ f(\bm S_n^{(1)}) \} \Exp \{ f(\bm S) \} + \big[ \Exp \{ f(\bm S) \} \big] ^2.
\end{multline*}
Since $\bm S_n^{(1)}$ and $\bm S_n^{(2)}$ are identically distributed and conditionally independent given $\bm X_n$, the first term on the right-hand side can be written as 
\begin{equation*}
%\label{eq:expect}
\Exp \big[ \Exp \{ f(\bm S_n^{(1)}) f(\bm S_n^{(2)}) \mid \bm X_n\}  \big] 
= 
\Exp \big\{ f(\bm S_n^{(1)}) f(\bm S_n^{(2)}) \big\} .
\end{equation*}
The function $(x,y) \mapsto f(x) f(y)$ being bounded and continuous, the convergence in $(a)$ implies that, as $n \to \infty$,
\[
\textstyle \Exp \Big\{ \Big(\int f d \Pr^{ \bm S_n^{(1)} \mid \bm X_n} -  \int f dQ \Big)^2 \Big\}
\to
\Exp \big[ f(\bm S^{(1)}) f(\bm S^{(2)}) \big] -  2   \Exp \{ f(\bm S^{(1)}) \} \Exp \{ f(\bm S) \} + \big[ \Exp \{ f(\bm S) \} \big] ^2 = 0,
\]
where $\bm S^{(1)}$ and $\bm S^{(2)}$ are independent copies of $\bm S$.

$(c) \Rightarrow (b)$: by Lemma~\ref{lem:bl1} and Corollary 1.4.5 in \cite{VanWel96}, it suffices to show that, as $n \to \infty$, 
\[
\Exp \{ f_1(\bm S_n^{(1)}) \cdots f_M(\bm S_n^{(M)}) \} \to \prod_{j=1}^M \Exp \{ f_j(\bm S) \}
\]
for any $f_1, \dots, f_M$ bounded and Lipschitz continuous. By independence of $\bm S_n^{(1)},\dots,\bm S_n^{(M)}$ conditionally on $\bm X_n$, we can write the left-hand side as
\[
\Exp \big[ \Exp \{ f_1(\bm S_n^{(1)}) \cdots f_M(\bm S_n^{(M)})  \mid \bm X_n \} \big] 
=
\Exp\big[ \Exp\{f_1(\bm S_n^{(1)}) \mid \bm X_n\} \cdots \Exp\{ f_M( \bm S_n^{(M)}) \mid \bm X_n\} \big],
\]
and the assertion follows from $(c)$, Lemma~\ref{lem:bl1c} and dominated convergence for convergence in probability. 

$(c) \Leftrightarrow (d)$: fix $f$ bounded and Lipschitz continuous and $\eps>0$, and denote by $K$ a bound on~$f$. Then, for any $n\in\N$,
\begin{align} \label{eq:cdc}
\nonumber \Pr\Big\{ \big| \textstyle \int f d\hat \Pr_M^{\bm S_n} - \int f d\Pr^{\bm S_n^{(1)} \mid \bm X_n} \big| \ge \eps \Big\} 
& =
\nonumber \Ex \Big[ \Pr \Big\{ \big| \textstyle \int f d\hat \Pr_M^{\bm S_n}  - \Ex f(\bm S_n^{(1)}) \big| \ge \eps \, \big| \, \bm X_n \Big\}  \Big] \\
&\le
\frac{1}{\eps^2 M^2} \Ex \Big[  \Var \big\{ \textstyle \sum_{i=1}^M f(\bm S_n^{(i)}) \, \big| \, \bm X_n \big\}  \Big]  \le \displaystyle \frac{K}{\eps^2 M}
\end{align}
by Chebychev's inequality. As a consequence, 
$$
\int f d\hat \Pr_M^{\bm S_n} - \int f d\Pr^{\bm S_n^{(1)} \mid \bm X_n} \p 0
$$
as $n,M \to \infty$ since the upper bound on the right-hand side of~\eqref{eq:cdc} is independent of $n$. The equivalence $(c) \Leftrightarrow (d)$ is then a consequence of Lemma~\ref{lem:bl1c}.

Finally, if $\Db=\R^d$ and if the d.f.\ of $Q$ is continuous, the equivalences 
$(c) \Leftrightarrow (e)$ and  $(d) \Leftrightarrow (f)$ are immediate consequences of Lemma~\ref{lem:bl1c}.
\end{proof}

Lemma~\ref{lem:uncond:cond} arises finally as a simple corollary of Lemma~\ref{lem:uncond:cond2}. 

\begin{proof}[Proof of Lemma~\ref{lem:uncond:cond}]
Denote the assertions $(a)$--$(f)$ in Lemma~\ref{lem:uncond:cond2} by $(a')$--$(f')$, respectively. Let $Q=\Pr^{\bm S}$ and note that, by Lemma~\ref{lem:bl1}, $\bm S_n \leadsto \bm S$ implies that $d_{\BL}(\Pr^{ \bm S_n}, Q)\to0$. Then, the triangle inequality and Lemma~\ref{lem:uncond:cond2} immediately imply the equivalences $(c)\Leftrightarrow(c')\Leftrightarrow(d')\Leftrightarrow(d)$. 

Since $(b)\Rightarrow (a) \Rightarrow (a') \Leftrightarrow (c')$, to show the equivalence between $(a)$--$(d)$, it remains to be shown that $(c')$ implies $(b)$. By Corollary 1.4.5 in \cite{VanWel96}, it suffices to show that 
\[
\Exp \{ f_0(\bm S_n) f_1(\bm S_n^{(1)}) \cdots f_M(\bm S_n^{(M)}) \} \to \prod_{j=0}^M \Exp \{ f_j(\bm S) \}
\]
for any $f_0, \dots, f_M$ bounded and Lipschitz continuous. By independence of $\bm S_n^{\scs (1)},\dots,\bm S_n^{\scs (M)}$ conditionally on $\bm X_n$, we obtain that
\[
\Exp\big[ \Exp \{ f_0(\bm S_n) f_1(\bm S_n^{(1)}) \cdots f_M(\bm S_n^{(M)}) \mid \bm X_n \} \big] 
=
\Exp\big[ f_0(\bm S_n) \Exp \{ f_1(\bm S_n^{(1)}) \mid \bm X_n \} \cdots \Exp\{ f_M(\bm S_n^{(M)}) \mid \bm X_n\} \big],
\]
and the assertion follows from $(c')$, Lemma~\ref{lem:bl1c} and dominated convergence for convergence in probability.

Finally, if $\Db=\R^d$ and if the d.f.\ of $Q=\Pr^{\bm S}$ is continuous, $(c') \Leftrightarrow (e')$ by Lemma~\ref{lem:uncond:cond2} and the equivalences $(e)\Leftrightarrow(e')\Leftrightarrow(f')\Leftrightarrow(f)$ follow from the fact that $d_K(\Pr^{ \bm S_n}, Q)\to 0$ (a consequence of Lemma~\ref{lem:bl1}), the triangular inequality and Lemma~\ref{lem:uncond:cond2}.
\end{proof}

\section{Extension to stochastic processes with bounded sample paths}
\label{sec:extension}

As in the previous section, let $\bm X_n$ be some data formally seen as a random variable in some measurable space~$\X_n$. Furthermore, let $T$ denote an arbitrary non-empty set and let $\ell^\infty(T)$ denote the set of real-valued bounded functions on $T$ equipped with the supremum distance. Since, as already mentioned in the introduction, the latter metric space is in general neither separable nor complete, one cannot typically set $\Db = \ell^\infty(T)$ and apply the results of the previous section.

To remedy this shortcoming, we are hereafter specifically interested in the situation in which the $\Db$-valued statistic $\bm S_n$ of the previous section is a stochastic process $\G_n=\G_n(\bm X_n)$ on $T$ constructed from $\bm X_n$. It is assumed that every \emph{sample path} $t \mapsto \G_n(t,\bm X_n(\omega))$ is a bounded function so that $\G_n$ may formally be regarded as a map from the underlying probability space $\Omega$ into $\ell^\infty(T)$ without however imposing any measurability conditions. We additionally suppose that, as $n \to \infty$, $\G_n$ converges weakly in $\ell^\infty(T)$ to some tight, Borel measurable stochastic process $\G$ in the sense of Hoffmann-J\o rgensen \citep[see, e.g.,][Section~1.3]{VanWel96} (which in fact implies that $\G_n$ is asymptotically measurable). Extending the setting of Section~\ref{sec:equiv:cond:uncond}, we further assume that $\G_n^{\scs (1)} = \G_n^{\scs (1)}(\bm X_n, \bm W_n^{\scs (1)}),\G_n^{\scs (2)} = \G_n^{\scs (2)}(\bm X_n, \bm W_n^{\scs (2)}), \dots$ are \emph{bootstrap replicates} of $\G_n$, that is, stochastic processes on $T$ depending on additional identically distributed random variables $\bm W_n^{\scs (1)}, \bm W_n^{\scs (2)}, \dots$ in some measurable space $\W_n$ that can, in many cases, be interpreted as \emph{bootstrap weights} and should in general be seen as the additional sources of randomness introduced by the resampling scheme. As for $\G_n$, it is assumed that the sample paths of $\G_n^{\scs (1)}, \G_n^{\scs (2)}, \dots$ also belong to $\ell^\infty(T)$ and, when seen as maps into $\ell^\infty(T)$, no measurability assumptions are made on these bootstrap replicates either. When $\bm X_n$ represents i.i.d.\ observations and $\G_n$ is the general empirical process constructed from $\bm X_n$, several examples of possible bootstrap replicates of $\G_n$ can for instance be found in \citet[Section~3.6]{VanWel96}. As in Section~3.6 of the latter reference, we assume throughout this section that the underlying probability space is independent of $n$ and has a product structure, that is, $\Omega=\Omega_0 \times \Omega_1 \times \cdots$ with probability measure $\Pr=\Pr_0 \otimes \Pr_1\otimes \cdots$, where $\Pr_i$ denotes the probability measure on~$\Omega_i$, such that, for any $\omega \in \Omega$, $\bm X_n(\omega)$ only depends on the first coordinate of $\omega$ and $\bm W_n^{\scs (i)}(\omega)$ only depends on the $(i+1)$-coordinate of $\omega$, implying in particular that $\bm X_n, \bm W_n^{\scs (1)}, \bm W_n^{\scs (2)}, \dots$ are independent.  

Some additional notation is needed before our main result can be stated. For any map $Z:\Omega \to \R$,  let $Z^*$ be any \emph{minimal measurable majorant of $Z$ with respect to $\Pr$}, that is, $Z^*:\Omega \to [-\infty,\infty]$ is measurable, $Z^* \ge Z$ and $Z^* \le U$ almost surely for any measurable function $U:\Omega \to [-\infty,\infty]$ with $U \ge Z$ almost surely. A \emph{maximal measurable minorant of $Z$ with respect to $\Pr$} is denoted by $Z_*$ and defined by $Z_* = -(-Z)^*$ \citep[see][Section~1.2]{VanWel96}. Furthermore, for any $i \in \{0,1,\dots\}$, we define the map $Z^{i*}:\Omega \to [-\infty,\infty]$ such that, for any $(\omega_0,\dots,\omega_{i-1},\omega_{i+1},\dots) \in \Omega_0 \times \dots \Omega_{i-1} \times \Omega_{i+1} \times \cdots$, the map $\omega_i \mapsto Z^{i*}(\omega_0,\dots,\omega_{i-1},\omega_i,\omega_{i+1},\dots)$ is a minimal measurable majorant of  $\omega_i \mapsto Z(\omega_0,\dots,\omega_{i-1},\omega_i,\omega_{i+1},\dots)$ with respect to $\Pr_i$. Finally, for a real-valued function $Y$ on $\mathcal X_n \times \mathcal W_n$ such that $\bm w \mapsto Y(\bm x, \bm w)$ is measurable for all $\bm x \in \X_n$, we further use the notation
\[
 \Ex (Y \mid \bm X_n) = \int_{\mathcal W_n} Y(\bm X_n, \bm w) \, d \Pr^{\bm W_n^{(i)}}(\bm w),
\]
provided the integral exists. Note that if $Y$ is jointly Borel measurable, the right-hand side of the last displays defines a version of the conditional expectation of $Y$ given $\bm X_n$, whence the notation.

\begin{lem} 
\label{lem:proc}
With the previous notation and under the above assumptions, the following three assertions are equivalent:
\begin{compactenum}[(a)]
\item
As $n\to\infty$,
\begin{equation}
\label{eq:jointweakconv2}
(\bm \G_n, \bm \G_{n}^{(1)}, \bm \G_{n}^{(2)}) \leadsto (\bm \G, \bm \G^{(1)}, \bm \G^{(2)}) \qquad \text{in } \{\ell^\infty(T)\}^{3},
\end{equation}
where $\bm \G, \bm \G^{(1)},  \bm \G^{(2)}$  are i.i.d.
\item
For any $M\ge 2$,  as $n\to\infty$,
\begin{equation}
\label{eq:jointweakconvM}
(\bm \G_n, \bm \G_{n}^{(1)}, \dots, \bm \G_{n}^{(M)}) \leadsto (\bm \G, \bm \G^{(1)}, \dots, \bm \G^{(M)}) \qquad \text{in } \{\ell^\infty(T)\}^{M+1},
\end{equation}
where $\bm \G, \bm \G^{(1)}, \dots, \bm \G^{(M)}$  are i.i.d.

\item  As $n\to\infty$,
\begin{equation}
\label{eq:BL1conv}
\sup_{h \in \BL_1(\ell^\infty(T))} \Big| \Ex \{ h(\G_n^{(1)})^{1*} \mid \bm X_n \} - \Ex \{ h(\G) \} \Big| \op 0,
\end{equation}
and $\G_n^{\scs (1)}$ is asymptotically measurable, where $\op$ denotes convergence in outer probability. 
\end{compactenum}
\end{lem}

\noindent Let us make a few comments on this result: 
\begin{compactitem} 
\item Assertion~$(c)$ is the extension put forward by \cite{GinZin90} of the conditional formulation of bootstrap consistency in a separable metric space $\Db$ to the non-necessarily separable space $\ell^\infty(T)$. Section~3.6 in \cite{VanWel96} and Chapter~10 in \cite{Kos08} in particular provide proofs of Assertion~$(c)$ for various bootstraps of the general empirical process constructed from i.i.d.\ observations along with continuous mapping theorems \emph{for the bootstrap} and a functional delta method \emph{for the bootstrap} that can be used to transfer \eqref{eq:BL1conv} to the statistic level in certain situations.

\item In \cite{VanWel96}, \cite{Van98} and \cite{Kos08}, the expression on the left-hand side of \eqref{eq:BL1conv} appears without the minimal measurable majorant with respect to the ``weights''. This is a consequence of the fact that, for all the resampling schemes considered in these monographs, the function $\bm w \mapsto \G_n^{\scs (1)}(\bm x, \bm w)$ is continuous for all $\bm x \in \X_n$, implying that $\bm w \mapsto h\{\G_n^{\scs (1)}(\bm x, \bm w)\}$ is measurable for all $\bm x \in \X_n$ and all $h \in \BL_1(\ell^\infty(T))$. However, the minimal measurable majorant  becomes for instance necessary if one wishes to apply Lemma~\ref{lem:proc}  to certain stochastic processes appearing when using the \emph{parametric bootstrap} (e.g., for goodness-of-fit testing, see, \citealp{StuGonPre93,GenRem08}). To see this, suppose that $\bm X_n$ is an i.i.d.\ sample of size $n$ from some d.f.\ $G$ on the real line, with $G$ from some parametric family $\{G_\theta\}$. A natural stochastic process, from which one may for instance construct classical goodness-of-fit statistics, is then $\G_n(t) = \sqrt{n} \{G_n(t) - G(t)\}$, $t \in \R$, where $G_n$ is the empirical d.f.\ of $\bm X_n$. Bootstrap samples are generated by sampling from $G_{\theta_n}$, where $\theta_n = \theta_n(\bm X_n)$ is an estimator of $\theta$. 
Note in passing that the latter way of proceeding is compatible with the product-structure condition on the underlying probability space since bootstrap samples can equivalently be regarded as obtained by applying $G_{\scs \theta_n}^{-1}$ component-wise to independent random vectors $\bm W_n^{\scs (1)},\bm W_n^{\scs (2)},\dots$ independent of $\bm X_n$ and whose components are i.i.d.\ standard uniform.  Now, corresponding parametric bootstrap replicates of $\G_n$ are given by $\G_n^{\scs (i)} = \sqrt{n} (G_n^{\scs (i)} -G_n)$, where $G_n^{\scs (i)}$ is the empirical d.f.\ of the sample $(G_{\theta_n}^{\scs -1}(W_{n1}^{\scs (i)}),\dots,G_{\theta_n}^{\scs -1}(W_{nn}^{\scs (i)}))$.
The need for the minimal measurable majorant with respect to the ``weights'' in~\eqref{eq:BL1conv} is then a consequence of the fact that the function from $\R^n$ to~$\R$ defined by
$$
\bm w^{\scs(i)} \mapsto h \{ \G_n^{(i)} (\bm x,\bm w^{\scs(i)}) \} = h \bigg( \frac{1}{\sqrt{n}} \sum_{j=1}^n \left[ \1 \{G_{\theta_n(\bm x)}^{-1}(w_j^{\scs(i)}) \leq \cdot \} - \1(x_j \le \cdot) \right] \bigg)
$$ 
is not measurable for all $h \in \BL_1(\ell^\infty(\R))$ and all $\bm x \in \X_n$, as can for instance be verified by adapting arguments from \citet[Section~15]{Bil99}.

\item Bootstrap asymptotic validity in the form of Assertions~$(a)$ or~$(b)$ is less frequently encountered in the literature, although, as discussed in the introduction, it may be argued that this unconditional formulation is more intuitive and easy to work with. It is proved for example in  \cite{GenRem08} (for $M=1$), \cite{RemSca09}, \cite{Seg12}, \cite{GenNes14}, \cite{BerBuc17} and \cite{BucKoj16, BucKoj16b}, among many others, for various stochastic processes arising in statistical tests on copulas or for assessing stationarity.

\item As mentioned in the introduction, note that Assertions~$(b)$ and~$(c)$ are known to be equivalent for the special case of the \emph{multiplier CLT} for the general empirical process based on i.i.d.\ observations and, in this case, it is even sufficient to consider $M=1$ in~$(b)$: Corollary 2.9.3 in \cite{VanWel96} corresponds to Assertion~$(b)$, while Theorem~2.9.6 corresponds to Assertion~$(c)$. The equivalence between the two follows by combining Theorem~2.9.6 with Theorem~2.9.2.  
\end{compactitem}

Before proving Lemma~\ref{lem:proc}, we provide a useful corollary which is an immediate consequence of Lemma~\ref{lem:proc} and Lemma~\ref{lem:uncond:cond}. It may be regarded as an analogue of Theorem 1.5.4 in \cite{VanWel96} in a conditional setting and, roughly speaking, states that conditional weak convergence of a sequence of stochastic processes is equivalent to the conditional weak convergence of finite-dimensional distributions and (unconditional) asymptotic tightness.

\begin{cor}
%\label{cor:fidi:marg:tight}
Suppose that the assumptions of Lemma~\ref{lem:proc} are met. Then, any of the equivalent assertions in that lemma is equivalent to the fact that the finite dimensional distributions of $\G_n^{\scs (1)}$ conditionally weakly converge to those of $\G$ in probability, that is, for any $k\in\N$ and $s_1, \dots, s_k\in T$, 
\begin{align}
\label{eq:dBL:fidi}
d_\BL\left(\Pr^{(\G_n^{(1)} (s_1), \dots,  \G_n^{(1)} (s_k)) \mid \bm X_n}, \Pr^{(\G (s_1), \dots, \G (s_k))} \right)  
\p 0
\end{align}
as $n\to \infty$, and  that $\G_n^{\scs (1)}$ is (unconditionally) asymptotically tight.
\end{cor}

\begin{proof}[Proof of Lemma~\ref{lem:proc}] 
We closely follow the proof of Theorem 2.9.6 of \cite{VanWel96} and rely on Lemma~\ref{lem:uncond:cond} when necessary.

$(b) \Rightarrow (a)$: trivial.

$(a)\Rightarrow (c)$: Asymptotic measurability  of $\G_n^{\scs (1)}$ is an immediate consequence of the weak con\-vergence of $\G_n^{\scs (1)}$ to $\G^{(1)}$ in $\ell^\infty(T)$ \citep[Lemma 1.3.8]{VanWel96}. Next, by Theorems 1.5.4 and 1.5.7 in \cite{VanWel96}, the latter convergence implies that there exists a semimetric $\rho$ on $T$ such that $(T,\rho)$ is totally bounded and such that, for any $\eps > 0$,
\begin{align} 
\label{eq:tight}
\lim_{\delta \downarrow 0}\limsup_{n\to\infty} \Pr^* \Big\{ \sup_{\rho(s,t)<\delta} |\G_n^{(1)}(s)- \G_n^{(1)}(t)|> \eps \Big\} = 0.
\end{align}
Fix $\ell\in\N$. For any $s \in T$, let $B(s, 1/\ell) = \{t \in T : \rho(s,t) < 1/\ell \}$ denote the ball of radius $1/\ell$ centered at $s$. Since $(T,\rho)$ is totally bounded, there exists $k=k(\ell) \in \N$ and $s_i = s_i(\ell)\in T$, $i \in \{1,\dots,k\}$, such that $T$ is included in the union of all balls $B(s_i, 1/\ell)$, $i \in \{1,\dots,k\}$. The latter allows us to define a mapping $\Pi_\ell:T \to T$ defined, for any $s \in T$, by $\Pi_\ell(s) = s_{i^*}$ where $s_{i^*}$ is the center of a ball containing $s$. 
Now, to prove~\eqref{eq:BL1conv}, we consider the decomposition 
$$
\sup_{h \in \BL_1(\ell^\infty(T))} \Big |\Ex \{ h(\G_n^{(1)})^{1*} \mid \bm X_n \} - \Ex \{ h(\G) \} \Big| \leq I_n(\ell) + J_n(\ell) + K(\ell), \qquad \ell \in \N, 
$$
where
\begin{align*}
I_n(\ell) &= \textstyle \sup_{h \in \BL_1(\ell^\infty(T))} \Big|  \Ex\{ h(\G_n^{(1)})^{1*} \mid \bm X_n\} - \Ex\{ h( \G_n^{(1)}  \circ \Pi_\ell )^{1*} \mid \bm X_n\}  \Big|, \\
J_n(\ell) &=\textstyle \sup_{h \in \BL_1(\ell^\infty(T))} \Big|\Ex\{ h( \G_n^{(1)}  \circ \Pi_\ell )^{1*} \mid \bm X_n\} - \Ex\{ h( \G \circ \Pi_\ell ) \} \Big|, \\
K(\ell) &= \textstyle\sup_{h \in \BL_1(\ell^\infty(T))} \Big| \Ex\{ h( \G \circ \Pi_\ell ) \} - \Ex\{ h(\G)  \} \Big|.
\end{align*}
Some thought reveals that~\eqref{eq:BL1conv} is proved if, for any $\eps > 0$, 
\begin{equation}
\label{eq:limIn}
\lim_{\ell\to\infty} \limsup_{n\to\infty} \Pr^* \left\{ I_n(\ell) > \eps \right\} = 0,
\end{equation}
and similarly for $J_n(\ell)$ and $K(\ell)$.

\emph{Term $I_n(\ell)$:} By Markov's inequality for outer probabilities (Lemma 6.10 in \citealp{Kos08}), it suffices to show \eqref{eq:limIn} with  $\Pr^* \left\{ I_n(\ell) > \eps \right\}$ replaced by $\Ex^* \{ I_n(\ell) \}$. For any $\ell \in \N$, we have, by Lemma~1.2.2 (iii) in \cite{VanWel96},
\begin{align*}
I_n(\ell) &\le \sup_{h \in \BL_1(\ell^\infty(T))} \Ex \Big\{ |  h( \G_n^{(1)}  \circ \Pi_\ell ) - h( \G_n^{(1)}) |^{1*} \mid \bm X_n  \Big\} \\
&\le 
\Ex\Big[ \Big\{ \sup_{s\in T} |  \G_n^{(1)}  \circ \Pi_\ell (s) - \G_n^{(1)} (s) | \wedge 1 \Big\}^* \mid \bm X_n\Big]  \le  \Ex\Big\{ L_n(\ell)^*  \mid \bm X_n\Big\},
\end{align*}
where $\wedge$ denotes the minimum operator and $L_n(\ell) = \sup_{\rho(s,t) < 1/\ell} |\G_n^{\scs (1)}(s)- \G_n^{\scs (1)}(t)| \wedge 1$. It follows that $\Ex^* \{ I_n(\ell) \} \leq \Ex \{ L_n(\ell)^* \}$. Note that, by Lemma 1.2.2 (viii) in \cite{VanWel96}, we may choose $L_n(\ell)^*$ in such a way that $\ell \mapsto L_n(\ell)^*$ is nonincreasing almost surely. Then $\ell \mapsto \Pr \{ L_n(\ell)^* > \eps\}$ is nonincreasing as well, and  from~\eqref{eq:tight} and Problem 2.1.5 in \citet[see also Section 2.1.2]{VanWel96}, we have that $L_n(\ell_n)^* \to 0$ in probability as $n \to \infty$ for any sequence $\ell_n \to \infty$, 
which, by dominated convergence for convergence in probability, implies that $\Ex \{ L_n(\ell_n)^* \} \to 0$. Hence, $\lim_{\ell\to\infty} \limsup_{n\to\infty} \Ex[L_n(\ell)^*]=0$ by invoking Problem~2.1.5 in \cite{VanWel96} again.

\emph{Term $J_n(\ell)$:} Fix $\ell \in \N$ and recall that the centers of the balls defining $\Pi_\ell$ were denoted by $s_1,\dots,s_k$. Since the weak convergence stated in~\eqref{eq:jointweakconv2} implies weak convergence of the respective finite dimensional distributions, we may invoke the equivalence between $(a)$ and $(c)$ in Lemma~\ref{lem:uncond:cond} to conclude (with the help of the triangular inequality and Lemma~\ref{lem:bl1}) that~\eqref{eq:dBL:fidi} holds, that is, that
\begin{equation}
\label{eq:Bn}
B_n(s_1,\dots,s_k) = \sup_{h\in \BL_1(\R^k)} \Big| \Ex[ h \{ \G_{n}^{(1)} (s_1), \dots,  \G_{n}^{(1)} (s_k) \} \mid \bm X_n]- \Ex[ h \{ \G (s_1), \dots,  \G(s_k) \} ] \Big | \p 0.
\end{equation}
Next, let $h \in \BL_1(\ell^\infty(T))$ be arbitrary. Define $f: \R^k \to \ell^\infty(T)$ such that, for any $\bm x \in \R^k$ and $s \in T$, $f(\bm x)(s) = x_i$ if $\Pi_\ell (s) = s_i$. Furthermore, let $g: \R^k \to \R$ be defined as $g(\bm x) = h ( f(\bm x) )$ implying that $h( \G \circ \Pi_\ell) = g(\G(s_1), \dots, \G(s_k))$. Some thought reveals that $g \in \BL_1(\R^k)$, whence $J_n(\ell) \le B_n(s_1,\dots,s_k) \to 0$ in probability as $n\to\infty$ for all $\ell \in \N$, implying the analogue of~\eqref{eq:limIn} for $J_n(\ell)$.

\emph{Term $K(\ell)$:} For any $\ell \in \N$, we have 
\[
K(\ell) \le \Ex \Big\{ \sup_{s\in T} |  \G \circ \Pi_\ell (s) - \G(s) |\wedge 1 \Big\}
\le 
\Ex \Big\{ \sup_{\rho(s,t) < 1/\ell} |\G(s) - \G(t)| \wedge 1 \Big\}.
\]
By tightness of $\G$, Addendum 1.5.8 in \cite{VanWel96}  and dominated convergence, the expectation on the right converges to zero as $\ell \to \infty$, implying the analogue of~\eqref{eq:limIn} for $K(\ell)$.

$(c)\Rightarrow (b)$:
To prove~\eqref{eq:jointweakconvM}, we need to show the weak convergence of the finite-dimensional distributions and marginal asymptotic tightness. We start with the former. 
Let $M,k \in \N$ and $s_1,\dots,s_k \in T$. It suffices to show that, as $n\to\infty$,
\begin{multline} \label{eq:fidi1}
\big(\G_n(s_1), \dots, \G_n(s_k), \G_n^{(1)}(s_1), \dots, \G_n^{(1)}(s_k), \dots, \G_n^{(M)}(s_1), \dots, \G_n^{(M)}(s_k)\big) \\
\leadsto
\big(\G(s_1), \dots, \G(s_k), \G^{(1)}(s_1), \dots, \G^{(1)}(s_k), \dots, \G^{(M)}(s_1), \dots, \G^{(M)}(s_k)\big)
\end{multline}
in $\R^{(M+1)k}$.
Now, for any $g\in \BL_1(\R^k)$, the function $h:\ell^\infty(T) \to \R$ defined by $h(f) = g(f(s_1), \dots, f(s_k))$ is an element of  $\BL_1(\ell^\infty(T))$. From~\eqref{eq:BL1conv}, we then obtain that~\eqref{eq:Bn} holds, or, equivalently, that~\eqref{eq:dBL:fidi} holds. We may hence invoke the equivalence between $(b)$ and $(c)$ in Lemma~\ref{lem:uncond:cond} to obtain~\eqref{eq:fidi1}.

It remains to show marginal tightness. Since $\G_n \leadsto \G$ in $\ell^\infty(T)$ and $\G_n^{\scs (1)},\dots,\G_n^{\scs (M)}$ are identically distributed, it is sufficient to show that $\G_n^{\scs (1)} \leadsto \G^{\scs (1)}$ in $\ell^\infty(T)$. Then, as in the proof of Theorem~2.9.6 of \cite{VanWel96}, for any $h \in \BL_1(\ell^\infty(T))$,
\begin{multline*}
| \Ex^* \{ h( \G_n^{(1)} ) \} - \Ex \{ h( \G^{(1)} ) \} |  
\leq 
\left| \Ex \Big[ \Ex \{  h( \G_n^{(1)} )^* \mid \bm X_n \} \Big] - \Ex^* \Big[ \Ex \{ h( \G_n^{(1)} )^{1*} \mid \bm X_n \} \Big] \right|  \\
+  
\left| \Ex^* \Big[ \Ex \{ h( \G_n^{(1)} )^{1*} \mid \bm X_n \}  - \Ex \{ h( \G^{(1)} ) \} \Big] \right|.
\end{multline*}
By dominated convergence for convergence in outer probability and~\eqref{eq:BL1conv}, the second term converges to zero. Since $h( \G_n^{\scs (1)} )^{1*} \ge \{h( \G_n^{\scs (1)} )_*\}^{1*} =h( \G_n^{\scs (1)} )_*$ almost surely,  the first term is bounded above by
\[
 \Ex \Big[ \Ex \{  h( \G_n^{(1)} )^* \mid \bm X_n \} \Big] - \Ex \Big[ \Ex \{  h( \G_n^{(1)} )_* \mid \bm X_n \} \Big] =  \Ex \{ h( \G_n^{(1)} )^* \} -  \Ex \{ h( \G_n^{(1)} )_* \}.
\]
The latter expression converges to zero since $\G_n^{\scs (1)}$ is assumed asymptotically measurable. The assertion follows from the Portmanteau Theorem \citep[see, e.g.,][Theorem 1.3.4~(i) and~(vii)]{VanWel96}.
\end{proof}

\section{Validity of bootstrap-based confidence intervals and tests}
\label{sec:conf:int:tests}

Whether the consistency of a resampling scheme is shown at the stochastic process level and then transferred to $\Db = \R^d$ or is directly proved at the statistic level, one naturally expects corresponding bootstrap-based confidence intervals and tests to be asymptotically valid. Specifically, the latter amounts to verifying that confidence intervals have the correct asymptotic coverage and that tests maintain their level asymptotically. To formally establish these expected consequences, in this section, we restrict ourselves to the classical situation of a real-valued statistic whose weak limit has a continuous distribution function.

\begin{cond}[$\R$-valued resampling mechanism] \label{cond:R}
Assume that Condition~\ref{cond:Db} holds with $\Db = \R$ and that, additionally, $\bm S_n$ converges weakly to a random variable $\bm S$ with continuous d.f.\ $F$.
\end{cond}

A result in the desired direction is for example Lemma~23.3 in \cite{Van98} and more specialized and deeper results are for instance collected in \citet[Sections~3.3 and 3.4]{Hor01}. Most results of that type do not however take into account the necessary approximation of the unobservable conditional d.f.\ of a bootstrap replicate by the empirical d.f.\ of a sample of bootstrap replicates. The following simple lemma does so and thus allows one to easily verify the asymptotic validity of bootstrap-based confidence intervals and tests constructed from a consistent resampling scheme in the sense of Lemma~\ref{lem:uncond:cond}.

As we continue, for $n,M \in \N$ and $x \in \R$, we use the following notation:
\[
F_n^M(x) =  \textstyle \frac1M \sum_{i=1}^M \1(\bm S_n^{(i)} \le x), \qquad 
F_n(x) = \Pr(\bm S_n^{(1)} \le x \mid \bm X_n) \quad \text{and} \quad
F(x) = \Pr(\bm S \le x). 
\]

\begin{lem} \label{lem:level}
Suppose that Condition~\ref{cond:R} is met and that one of the equivalent assertions in Lemma~\ref{lem:uncond:cond} holds. Then, for any $\alpha \in (0,1)$, 
\[
\lim_{n\to\infty} \Pr \{ \bm S_n \ge (F_n)^{-1}(1-\alpha) \} = \alpha \quad \text{ and } \quad \lim_{n,M\to\infty} \Pr \{ \bm S_n \ge (F_n^M)^{-1}(1-\alpha) \} = \alpha,
\]
where $G^{-1}$ denotes the generalized inverse of d.f.\ $G$, that is $G^{-1}(y) = \inf\{ x \in \R : G(x) \geq y \}$, $y \in (0,1]$. The statements with `$\ge$' replaced by `$>$' in the previous display hold as well.
\end{lem}

The assertion of this lemma involving conditional quantiles (or a version thereof) is usually provided in textbooks on the bootstrap to validate its use for the construction of confidence intervals and tests (see, e.g., Lemma~23.3 in \citealp{Van98}). For completeness, we shall prove it at the end of this section. The assertion involving empirical quantiles is the one to be used in practice as conditional quantiles are not available and must thus be approximated by Monte Carlo. Note in particular that the above formulation is general enough to {allow} $M=M(n)$ with $M(n) \to \infty$ as $n\to\infty$.

Let us now briefly verify that the asymptotic validity of bootstrap-based confidence intervals and tests is an immediate consequence of the preceding lemma. Start with the former and assume that $\bm S_n = \sqrt{n} (\theta_n - \theta)$, where $\theta_n$ is an estimator of some parameter $\theta \in \R$. Then, a natural confidence interval for $\theta$ is given by 
$$
I_{n,M,\alpha} = \Big[ \theta_n - n^{-1/2} (F_n^M)^{-1}(1-\alpha/2), \theta_n - n^{-1/2} (F_n^M)^{-1}(\alpha/2) \Big], \qquad \alpha \in (0,1/2).
$$
Note in passing that the above confidence interval is related to the so-called \emph{basic bootstrap confidence interval} \citep[see, e.g.,][Chapter 5]{DavHin97}. A consequence of Lemma~\ref{lem:level} is then that, if one of the equivalent assertions in Lemma~\ref{lem:uncond:cond} hold, $I_{n,M,\alpha}$ is of asymptotic level $1-\alpha$ in the sense that, as $n,M \to \infty$, 
$$
\Pr(\theta \in I_{n,M,\alpha}) = \Pr \{ \bm S_n \ge (F_n^M)^{-1}(\alpha/2) \} - \Pr \{ \bm S_n > (F_n^M)^{-1}(1-\alpha/2) \} \to 1 - \alpha.
$$

Let us now discuss the case of bootstrap-based tests. Assume that $\bm S_n$ is a test statistic for some null hypothesis $H_0$ such that large values of $\bm S_n$ provide evidence against $H_0$. It is then natural to reject $H_0$ at level $\alpha \in (0,1)$ when $\bm S_n > (F_n^M)(1 - \alpha)$. Should one of the equivalent assertions in Lemma~\ref{lem:uncond:cond} holds under $H_0$, Lemma~\ref{lem:level} immediately implies that this test holds it level asymptotically in the sense that, under $H_0$, $\Pr \{ \bm S_n \ge (F_n^M)^{-1}(1-\alpha) \} \to \alpha$ as $n,M \to \infty$.  If the bootstrap replicates are stochastically bounded under the alternative, then the test will also be consistent provided $\bm S_n$ converges to infinity in probability under the alternative.  

Under the same setting, another statistic of interest is
\[
p_n^M = \frac1M \sum_{i=1}^M \1(\bm S_n^{(i)} >\bm S_n) = 1 - F_n^M(\bm S_n),
\]
which may be interpreted as an approximate p-value for the test based on $\bm S_n$. The theoretical analogue of the latter is
\[
p_n = \Pr( \bm S_n^{(1)}> \bm S_n \mid \bm X_n) = 1 - F_n(\bm S_n).
\]
Intuitively, the resampling scheme being valid should imply that, under the null hypothesis, the statistics $p_{n}^M$ and $p_n$ are approximately standard uniform. The following result formalizes this.

\begin{cor}
\label{cor:p-values}
Suppose that Condition~\ref{cond:R} is met and that one of the equivalent assertions in Lemma~\ref{lem:uncond:cond} holds. Then, as $n\to\infty$,
\[
p_n \leadsto \mathrm{Uniform}(0,1)  \qquad \text{ and } \qquad  p_n^{M_n} \leadsto \mathrm{Uniform}(0,1),
\]
for any sequence $M_n \to \infty$ as $n\to\infty$.
\end{cor}

The proofs of Lemma~\ref{lem:level} and Corollary~\ref{cor:p-values} are given hereafter.

\begin{proof}[Proof of Lemma~\ref{lem:level}]
Consider the assertion involving conditional quantiles. Notice first that the weak convergence of $\bm S_n$ to $\bm S$, the continuity of $F$ and Lemma~\ref{lem:bl1} imply that
\begin{equation}
\label{eq:unif:SntoS}
d_K(\Pr^{\bm S_n}, \Pr^{\bm S}) = \sup_{x \in \R} |\Pr(\bm S_n \le x) - F(x) | \to 0 \qquad \text{as } n \to \infty.
\end{equation}
Next, combine Assertion~$(e)$ in Lemma~\ref{lem:uncond:cond} with~\eqref{eq:unif:SntoS} to obtain that every subsequence of $d_K(\Pr^{\bm S_n\ex{(1)} \mid \bm X_n}, \Pr^{\bm S})$ has a further subsequence along which this expression converges almost surely to zero as $n \to \infty$.  Let $\alpha \in (0,1)$ such that $F^{-1}$ is continuous at $1-\alpha$. As a consequence of Lemma 21.2 in \cite{Van98}, we obtain that $F_n^{-1}(1-\alpha) \as F^{-1}(1-\alpha)$ along that subsequence. Hence, the random vector $(\bm S_n, F_n^{-1}(1-\alpha))$ converges weakly to $(\bm S, F^{-1}(1-\alpha))$, again along that subsequence. Since $\Pr \{ \bm S = F^{-1}(1-\alpha) \} = 0$ by continuity, the Portmanteau Theorem implies that
\begin{equation}
\label{eq:condquant}
\Pr \{ \bm S_n\ge F_n^{-1}(1-\alpha) \} \to \Pr \{ \bm S \ge F^{-1}(1-\alpha) \} = \alpha
\end{equation}
along that subsequence. The latter equation holds for all expect at most countably many $\alpha \in (0,1)$. Because the left (resp.\ right) side of~\eqref{eq:condquant} is an increasing (resp.\ increasing continuous) function of $\alpha$,~\eqref{eq:condquant} must hold for all $\alpha \in (0,1)$. The first assertion follows since the subsequence we started with was arbitrary. Finally, note that one may replace `$\ge$' by `$>$' in the last display.

Consider the assertion involving empirical quantiles. Let $q_{1-\alpha}$ denote the $(1-\alpha)$-quantile of~$\bm S$. Since $\lim_{n\to \infty} \Pr(\bm S_n\ge q_{1-\alpha} ) = \Pr(\bm S \ge q_{1-\alpha}) = \alpha$ as a consequence of the Portmanteau Theorem and the continuity of $F$, it suffices to show that
\begin{multline*}
  \lim_{n,M \to \infty} |   \Pr \{ \bm S_n \ge  (F_n^M)^{-1}(1-\alpha) \} - \Pr(\bm S_n \ge q_{1-\alpha}) |  \\
  = \lim_{n,M \to \infty} | \Pr \{ F_n^M (\bm S_n) \ge {1-\alpha} \}  
  - \Pr \{ F(\bm S_n) \ge 1-\alpha \}| = 0,
\end{multline*}
where the equality follows from the fact that $F_n^{\scs M}$ and $F$ are right-continuous. Using the fact that, for any $a,b,x\in\R$ and $\eps>0$, $| \1(x \le a) - \1(x \le b) | \le \1(|x-a| \le \eps) +\1(|a-b|> \eps)$, we can estimate
\begin{multline*} 
  | \Pr \{ F_n^M (\bm S_n) \ge {1-\alpha} \} - \Pr \{ F(\bm S_n) \ge 1-\alpha \} |  \\
  \le   \Pr\{ |F(\bm S_n) - 1 + \alpha | \le \eps \} +   \Pr\{ | F(\bm S_n) - F_n^M (\bm S_n) | > \eps \}.
\end{multline*}
By the continuous mapping theorem and the Portmanteau Theorem, the first term on the right converges to $\Pr\{| F(\bm S) - 1+\alpha | \le \eps\}$ as $n \to \infty$, which can be made arbitrary small by decreasing~$\eps$. Combining Assertion~$(f)$ from Lemma~\ref{lem:uncond:cond} with~\eqref{eq:unif:SntoS} immediately implies that the second term converges to zero as $n,M \to \infty$, hence the first claim.

The claim with `$\ge$' replaced by `$>$' follows from the fact that, by continuity of $F$ and the Portmanteau Theorem, for any $x \in \R$, $\Pr(\bm S_n < x) \to \Pr(\bm S < x)$ as $n \to \infty$. The latter convergence can be made uniform by arguments as in Lemma 2.11 in \cite{Van98}, which, combined with~\eqref{eq:unif:SntoS} implies that $\sup_{x \in \R} \Pr(\bm S_n = x)$ converges to zero in probability as $n \to \infty$.
\end{proof}

\begin{proof}[Proof of Corollary~\ref{cor:p-values}] 
 The weak convergence $\bm S_n \leadsto \bm S$ as $n \to \infty$ together with the continuous mapping theorem implies that $1 - F(\bm S_n) \leadsto 1 - F(\bm S) \sim \mathrm{Uniform}(0,1)$ as $n \to \infty$. Combining Assertion~$(e)$ in Lemma~\ref{lem:uncond:cond} with~\eqref{eq:unif:SntoS}, we additionally immediately obtain that $F_n(\bm S_n) - F(\bm S_n)$ converges to zero in probability as $n \to \infty$, which implies that $p_n$ has the same weak limit as $1 - F(\bm S_n)$ as $n \to \infty$. Similarly, Assertion~$(f)$ in Lemma~\ref{lem:uncond:cond} combined with~\eqref{eq:unif:SntoS} readily implies that $p_n^{\scs M_n}$ has the same limit distribution as $1 - F(\bm S_n)$ as $n \to \infty$.
\end{proof}

\section{Concluding remarks}

\begin{figure}[t!]
\begin{center}
\includegraphics*[width=0.9\linewidth]{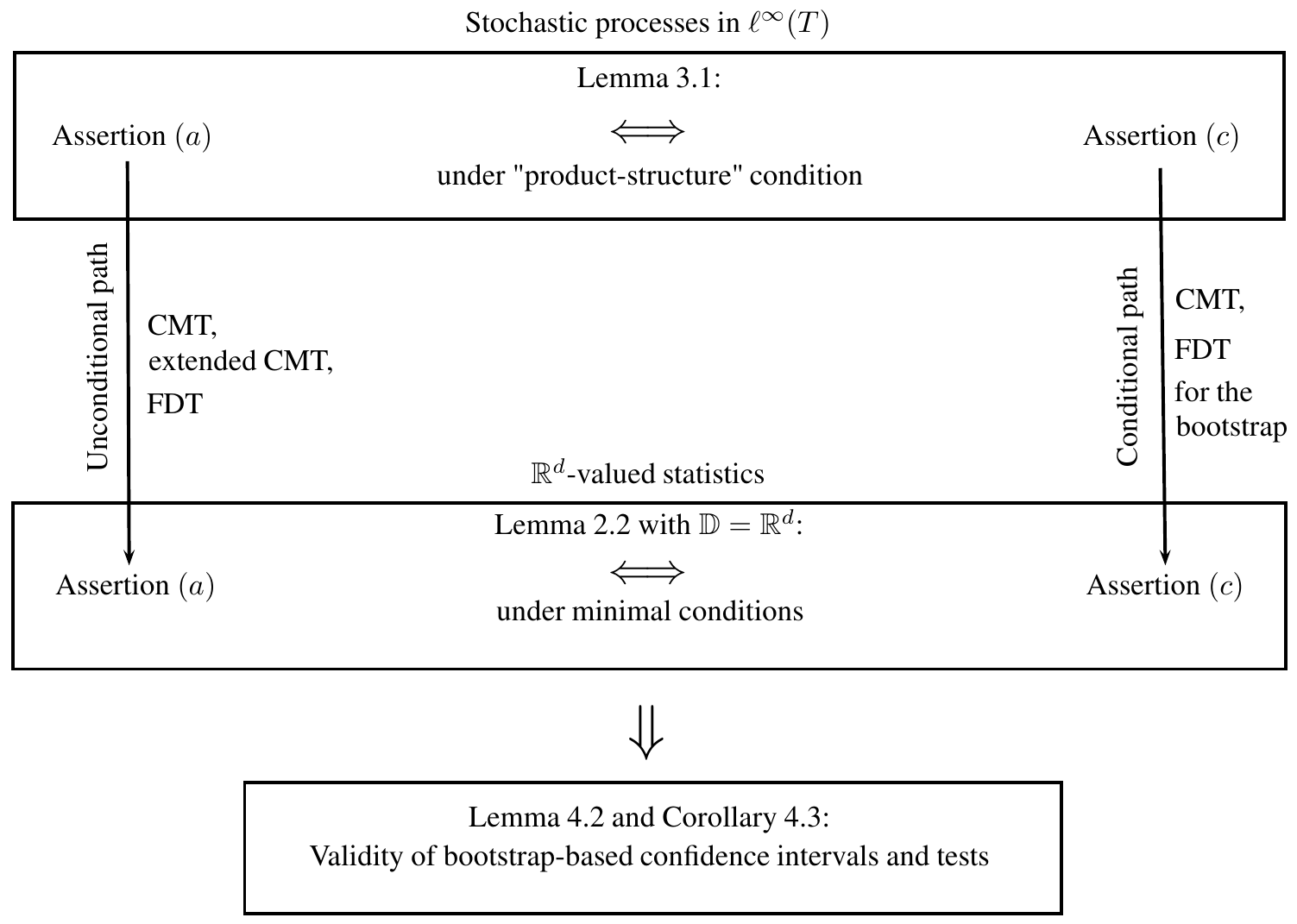}
\caption{\label{fig:summary} Summary of typical uses of the obtained results; CMT stands for ``continuous mapping theorem'' and FDT for ``functional delta method''.}
\end{center}
\end{figure}

As a picture often speaks better than words, we summarized in the diagram of Figure~\ref{fig:summary} the way the results obtained in this note could typically be used to prove the validity of bootstrap-based statistical inference procedures. From the point of view of applications of resampling schemes starting at the stochastic process level, the diagram highlights two paths to proving the asymptotic validity of bootstrap-based confidence intervals and tests: an unconditional path starting at Assertion~$(a)$ of Lemma~\ref{lem:proc} and a conditional path starting at Assertion~$(c)$ of Lemma~\ref{lem:proc}.

We conclude by summarizing the main consequences and features of the results obtained in this note, some of which explicitly appear in the diagram of Figure~\ref{fig:summary}: 

\begin{compactitem}
\item At the stochastic process level, it may be argued that one needs to deal with less subtle mathematical concepts to prove unconditional bootstrap consistency than to show its conditional version. Roughly speaking, the unconditional approach avoids the need to work with the seemingly awkward notion of ``conditional law'' of a non-measurable function. 

\item Focusing for instance on existing continuous mapping theorems \emph{for the bootstrap} \cite[][Section~10.1.4]{Kos08}, it appears that, for transferring Assertion~$(c)$ of Lemma~\ref{lem:proc} into Assertion~$(c)$ of Lemma~\ref{lem:uncond:cond}, more assumptions than just continuity of the underlying functional are necessary, thereby suggesting that the unconditional formulation of bootstrap consistency might be slightly more useful. Additionally, Assertion~$(a)$ of Lemma~\ref{lem:proc} can be combined with the \emph{extended continuous mapping theorem} \citep[][Theorem~1.11.1]{VanWel96}, while a version of the latter result for the bootstrap does not hitherto seem to~exist. 

\item The equivalence between the unconditional and the conditional formulation of bootstrap consistency at the stochastic process level only holds if the additional randomness in the bootstrap replicates is independent of the data (in fact, this assumption is only needed to make Assertion~$(c)$ well-defined). Interestingly enough, such a condition does not seem to be a restriction in practice as it seems satisfied by most if not all resampling schemes. 

\item Although, as already discussed, one cannot in general rely on Lemma~\ref{lem:uncond:cond} to deal with stochastic processes with bounded sample paths, this lemma remains general enough to deal with stochastic processes living in the Skorohod space (see, e.g., \citealp{Bil99}) since the latter can be metrized in such a way that it is separable and complete.

\end{compactitem}

\section*{Acknowledgments}

The authors are very grateful to an anonymous referee for making them aware of ``Hoeffding's trick'' \citep{Hoe52} and for several other very relevant suggestions which contributed to significantly {increasing} the scope of this note. The authors would also like to thank Jean-David Fermanian for fruitful discussions. This research has been supported by the Collaborative Research Center ``Statistical modeling of nonlinear dynamic processes'' (SFB 823) of the German Research Foundation, which is gratefully acknowledged. 

\bibliographystyle{chicago}
\bibliography{biblio}

\newpage
\thispagestyle{empty}

\begin{center}
{\LARGE Supplementary material for  \\ [2mm] ``A note on conditional versus joint unconditional weak \\ [2mm] convergence  in bootstrap consistency results''}
\vspace{1cm}

{\large Axel B\"ucher\footnote{Ruhr-Universit\"at Bochum,
Fakult\"at f\"ur Mathematik, 
Universit\"atsstr.~150, 44780 Bochum, Germany. 
{E-mail:} \texttt{axel.buecher@rub.de}} and Ivan Kojadinovic\footnote{CNRS / Universit\'e de Pau et des Pays de l'Adour, Laboratoire de math\'ematiques et applications -- IPRA, UMR 5142, B.P. 1155, 64013 Pau Cedex, France.
{E-mail:} \texttt{ivan.kojadinovic@univ-pau.fr}}

\vspace{.2cm} \today 
\vspace{1cm}
}

\end{center}

\begin{abstract}
This supplementary material contains the proofs of Lemmas~\ref{lem:bl1} and \ref{lem:bl1c}.
\end{abstract}

\appendix
\section{Proof of Lemma~\ref{lem:bl1}}

\begin{customlem}{\ref{lem:bl1}}
Suppose that $(\Db,d)$ is a separable metric space and let $P_\alpha$ be a net of probability measures on $(\Db,\mathcal D)$, where $\Dc$ denotes the Borel sigma field. Then $P_\alpha \leadsto P$ if and only if $d_\BL(P_\alpha,P) \to 0$. If $\Db=\R^d$ and if the d.f.\ of $P$ is continuous, we also have equivalence to $d_K(P_\alpha,P) \to 0$.
\end{customlem}

%\begin{proof}[Proof of Lemma~\ref{lem:bl1}] 
\begin{proof}
The assertion for $\Db=\R^d$ and if the d.f.\ of $P$ is continuous is, up to a slight generalization to nets, Lemma~2.11 in \cite{Van98}. The general case is, again up to a generalization to nets, Theorem 11.3.3 in \cite{Dud02}. Since we will need arguments from the latter proof in the proof of Lemma~\ref{lem:bl1c} below, we reproduce them for completeness.

We start by proving the claim under the additional assumption that $(\Db, d)$ is complete. By Ulam's Theorem (see, e.g., Theorem 1.3 in \citealp{Bil99}), the measure $P$ is tight. Hence, for any given $\eps>0$, we can find a compact set $K$ in $(\Db,d)$ such that $P(K) > 1-\eps$.  Let $K^\eps=\{ x \in \Db: d(x,y) < \eps \text{ for some } y \in K\}$ denote the $\eps$-enlargement of $K$. Furthermore, consider the function $g(x) = \max\{1-d(x,K) /\eps, 0\}$, $x \in \Db$, and note that $g$ is Lip\-schitz continuous and satisfies $\bm 1_K \le g \le \bm 1_{K^\eps}$. As a consequence of the Portmanteau Theorem (see, e.g., Theorem 1.3.4 in \citealp{VanWel96}), we obtain that
\[
\liminf P_\alpha(K^\eps) 
\ge 
\liminf \int g dP_\alpha 
\ge %=
\int gdP \ge P(K) >1-\eps.
\]

Next, consider the space $\Fc_K$ of functions $h=f|_K:K \to [-1,1]$ with $f\in \BL_1(\Db)$. 
It can be verified that $\Fc_K$ is equicontinuous and bounded, whence, by the Arzel\'a--Ascoli theorem, $\Fc_K$ is relatively compact in the Banach space $(C(K), \| \cdot \|_\infty)$. Relative compactness in a complete metric space is equivalent to total boundedness, so we can find functions $f_1, \dots, f_m \in \Fc_K$ such that, for any $f\in \Fc_K$, there exists $f_j$ with $\sup_{x\in K} |f(x) - f_j(x)| < \eps$.  As a consequence, by Lipschitz continuity of $f$ and $f_j$ on $\Db$, we have
\[
\sup_{x\in K^\eps} |f(x) - f_j(x)| < 3\eps.
\]
Finally, assembling bounds obtained so far, we obtain that, for any $f\in\BL_1(\Db)$,
\begin{align*}
\big| \textstyle \int f d(P_\alpha -P) \big| 
&=
\big| \textstyle \int (f-f_j) d(P_\alpha -P) +  \int f_j d(P_\alpha -P)\big| \\
&\le 
\textstyle \int \big|f-f_j\big| dP_\alpha +  \int \big|f-f_j\big| dP  +  \big|\int f_j d(P_\alpha -P)\big| \\
&\le 
\textstyle \int_{K^\eps} \big|f-f_j\big| dP_\alpha + 2 P_\alpha \{ (K^\eps)^c \} + \int_{K} \big|f-f_j\big| dP + 2 P(K^c)  +  \big|\int f_j d(P_\alpha -P)\big| \\
&\le 
3 \eps + 2 \eps  + \eps + 2 \eps  + \max_{j=1}^m\big| \textstyle \int f_j d(P_\alpha -P)\big|
\end{align*}
for sufficiently large $\alpha$. The Portmanteau Theorem implies that that maximum on the right-hand side converges to 0. The assertion follows since $\eps>0$ was arbitrary.

Finally, the case where $\Db$ is not complete can be treated analogously by passing to the completion of $\Db$; see, e.g., \cite{Dud02}.
\end{proof}

\section{Proof of Lemma~\ref{lem:bl1c}}

\begin{customlem}{\ref{lem:bl1c}}
Suppose that $(\Db,d)$ is a  separable metric space and let $(\hat P_\alpha)_\alpha$ denote a net of random probability measures on $(\Db,\Dc)$ defined on a probability space $(\Omega, \Ac, \Pr)$. Then, 
\begin{align} \label{eq:rw2}
\int fd\hat P_\alpha \p \int f dP 
\end{align}
for any $f$ bounded and Lipschitz continuous if and only if $d_\BL(\hat P_\alpha,P) \to 0$ in probability. Further, $d_\BL(\hat P_\alpha,P)$, considered as a map from $\Omega$ to $\R$, is measurable.

If $\Db=\R^d$ and if the d.f.\ of $P$ is continuous, then \eqref{eq:rw2} is also equivalent to $d_K(\hat P_\alpha,P) \to 0$ in probability, and $d_K(\hat P_\alpha,P)$ is measurable as well.
\end{customlem}

\begin{proof}
%\begin{proof}[Proof of Lemma~\ref{lem:bl1c}]
We only consider the case where $\Db$ is complete. For the general case, one can pass to the completion of $\Db$ as mentioned in the proof of Lemma~\ref{lem:bl1}. Sufficiency follows from linearity of integrals, since any bounded Lipschitz function can be scaled to a function in $\BL_1(\Db)$. Necessity  follows by carefully following the  proof of Lemma~\ref{lem:bl1}, which only made use of the fact that $P_\alpha \leadsto P$ implies $\int fdP_\alpha \to \int fdP$ for all $f$ bounded and Lipschitz continuous. More precisely, let $\delta>0$ and $\eta>0$. We need to show that $\Pr( d_\BL(\hat P_\alpha,P)>\delta) < \eta$ for all sufficiently large $\alpha$. Let $\eps=\delta/11$. Choose $K=K(\eps)$ and, subsequently, the functions $g, f_1, \dots, f_m$ as in the proof of Lemma~\ref{lem:bl1}. By assumption, we have $\Pr( A_\alpha )  < \eta$ for all sufficiently large $\alpha$, where
\[
A_\alpha =\Big\{ \textstyle \max\big\{ \big|\int g d(\hat P_\alpha-P)\big|,\big|\int f_1 d(\hat P_\alpha-P)\big|, \dots, \big|\int f_m d(\hat P_\alpha-P)\big| \big\}  > \eps\Big\}. 
\]
On the event $A_\alpha^c$, we can verify that $\hat P_\alpha(K^\eps) > 1 - 2 \eps$ and subsequently follow the proof of Lemma~\ref{lem:bl1} to establish that $\sup_{f\in \BL_1} \big| \int f d(\hat P_\alpha-P) \big| \le 11\eps=\delta$ for sufficiently large $\alpha$, which implies the assertion.

Let us next prove measurability of $d_\BL(\hat P_\alpha,P)$. Choose a countable dense subset $S$ of $\Db$.
Let $\Hc$ denote the set of all real-valued functions on $\Db$ of the form
\[
h(x) = q \max\{ 1-pd(x,s), 0 \}, \quad x \in \Db,
\]
where $s\in S$, $p, q \in \Q \cap [0, \infty)$ with $q\le 1$ and $pq \le 1$. Note that that $\Hc$ is a countable subset of $\BL_1(\Db)$, and that any nonnegative function  $f\in\BL_1(\Db)$ can be written as
\begin{align}  \label{eq:suph}
f(x) = \sup\{ h(x): h \le f, h \in \Hc\}.
\end{align}
Indeed, either a picture helps, or the following formal argument: for fixed $x\in \Db$ with $c=f(x)>0$, choose a sequence $(s_m)_m$ in $S$ converging to $x$ and let $\delta_m = d(x,s_m)$. Without loss of generality, we may assume that $\delta_m < c$ for all $m$. Choose $c_m \in \Q$ such that $0 < c_m < c-\delta_m$ and such that $c_m \to c$. Consider the function $h_m(z) = c_m \{ 1-c_m^{-1} d(z,s_m)\}^+$. Then $h_m\in \Hc$ and  $h_m(z) \le f(z)$ for all $z \in \Db$. Indeed, this bound is trivial for $z$ with $d(z,s_m) \ge c_m$, and otherwise, we have
\[
c-f(z) = f(x) - f(z) \le d(x,z) \le d(x,s_m) + d(s_m,z) = \delta_m + d(s_m,z),
\]
which implies that 
\[
f(z) \ge c - \delta_m - d(s_m,z) > c_m - d(s_m,z)  =  h_m(z).
\]
The assertion in \eqref{eq:suph} then follows from the fact that $h_m(x) = c_m - \delta_m \to c =f(x)$ as $m\to\infty$.

Next, let $\Hc_\vee$ denote the set of functions which are maxima of a finite number of functions in $\Hc$. Note that $\Hc_\vee \subset \BL_1(\Db)$. Further, let $\Hc_{\vee, -}$ denote the  the set of functions $h=h_1 - h_2$ with $h_1, h_2 \in \Hc_{\vee}$ such that $h\in \BL_1(\Db)$. Note that $\Hc_{\vee, -}$ is still countable and a subset of $\BL_1(\Db)$. 

Now, consider an arbitrary function $f\in\BL_1(\Db)$, and write $f=f^+ - f^-$ with $f^+ = \max(f,0)$ and  $f^- = \max(-f,0)$. The supremum in \eqref{eq:suph}, with $f$ replaced by $f^+$, is over countably many functions. Denote them by $h_1^+, h_2^+, \dots$ For $m\in \N$, let $g_m^+= \max_{j=1}^m h_j^+ \in \Hc_\vee$, such that $f^+= \lim_{m\to\infty} g_m^+$. Similarly, we can write $f^-$ as a limit of functions $g_m^- \in \Hc_\vee$. Note that $g_m^\pm$ can only be positive on the support of $f^\pm$. As a consequence, $g_m = g_m^+ - g_m^-$ is an element of $\Hc_{\vee, -}$. Indeed, if $x$ and $y$ are both in the support of $f^+$, then $|g_m(x) - g_m(y)| = |g_m^+(x) - g_m^+(y)| \le d(x,y)$. The  case where both are in the support of $f^-$  or where at least one of the points is in neither support is similar.  Finally, consider the case where $x$ is in the support of $f^+$ and $y$ is in the support of $f^-$, and without loss of generality assume that $g_m(x)>g_m(y)$. Then, $|g_m(x) - g_m(y)| = g_m(x) - g_m(y) = g_m^+(x) + g_m^-(y) \le f^+(x) + f^-(y) = f(x) - f (y) \le d(x,y)$.

This implies, by monotone convergence,
\[
\textstyle \big|\int f d(\hat P_\alpha - dP)\big| 
= 
\lim_{m\to\infty}  \big|\int g_m d(\hat P_\alpha - dP)\big| 
\le 
\sup_{g \in \Hc_{\vee, -}} \big|\int g d(\hat P_\alpha - dP)\big|.
\]
As a consequence, since $f$ was arbitrary and since $\Hc_{\vee, -} \subset \BL_1(\Db)$, we obtain that 
\[
d_{\BL}(\hat P_\alpha, P) = \textstyle \sup_{g \in \Hc_{\vee,-}} \big|\int g d(\hat P_\alpha - dP)\big|,
\] 
and the assertion follows since the supremum on the right-hand side is a countable supremum over measurable random variables. 

Finally, consider the assertions regarding the Kolmogorov distance. Note that $\hat P_\alpha((-\infty,\bm x])$, considered as map from $\Omega$ to $\R$ with $\bm x\in \R^d$ fixed, is measurable. Hence, by right-continuity of d.f.s, the Kolmogorov distance is measurable as well. To conclude, it is sufficient to show that the convergence in \eqref{eq:rw2} is equivalent to 
\begin{align*} %\label{eq:cdfc}
\hat P_\alpha((-\infty,\bm x]) \p P((-\infty, \bm x])
\end{align*}
for all $\bm x\in \R^d$; see, e.g, Problem~23.1 in \cite{Van98}.  This equivalence in turn follows by standard approximation arguments as in the proof of the classical Portmanteau Theorem; see, e.g., \cite{Van98}, Lemma~2.2. 
% Indeed, if the  convergence in \eqref{eq:cdfc} is true, then $\hat P_\alpha(I) \to P(I)$ in probability for all rectangles $I$, by continuity of the d.f.\ of $P$. Let $f$ be bounded (w.l.o.g.\ by 1) and Lipschitz continuous, and let $\eps>0$. Choose $K$ sufficiently large such that $P(K)<\eps$. Since $f$ is uniformly continuous on the compact set $K$, we may find a partition $K= \bigcup_j I_j$ of $K$ into finitely many rectangles $I_j$ such that $f$ varies at most $\eps$ on every $I_j$. Define $f_\eps(x)= \sum_{j} \inf f(I_j) \bm 1_{I_j}(x)$, such that $\sup_{x\in K} |f(x)-f_\eps(x)|\le \eps$. Then,
% \[
% \int |f-f_\eps | dP \le \eps + P(K^c) \le 2\eps, \qquad \int |f-f_\eps | d\hat P_\alpha \le \eps + P_\alpha(K^c) \le 2\eps
% \]
% for sufficiently large $\alpha$.
% Hence,
% \[
% \textstyle |\int f d(\hat P_\alpha-P)| \le 4 \eps + \sum_{j} | \hat P_\alpha(I_j) - P(I_j)  | \p 4 \eps.
% \]
% This implies \eqref{eq:rw2}, since $\eps>0$ was arbitrary.
%
% Next, assume \eqref{eq:rw2}. Every indicator $h=\bm 1_{(-\infty, \bm x]}$ can be approximated by two sequences of bounded Lipschitz functions $f_m$ and $g_m$ such that $0 \le f_m \uparrow h$ and $1 \ge g_m \downarrow h$. Hence, for any $\eps>0$
% \[
% \textstyle \hat P_\alpha((-\infty, \bm x]) \le \int g_m d \hat P_\alpha  \p \int g_m dP \le P((-\infty, \bm x]) + \eps
% \]
% for sufficiently large $m$. A similar lower bound applies, which proves the assertion.
\end{proof}

\end{document}